\documentclass[reqno]{amsart}
\usepackage[utf8]{inputenc}
\usepackage{amsmath,amsthm,amssymb}
\usepackage{tikz}
\usepackage{tikz-cd}

\usepackage[
    hmarginratio=1:1,
    textheight=10in,
    left=1in,
    right=1in
    ]{geometry}

\usepackage[english]{babel}

\usepackage[hidelinks]{hyperref}

\usepackage{wasysym, stackengine, makebox, graphicx}

\newtheorem{lemma}{Lemma}
\newtheorem{corollary}{Corollary}
\newtheorem{theorem}{Theorem}
\newtheorem{axiom}{Axiom}
\newtheorem{innercustomthm}{Axiom}
\newenvironment{customaxiom}[1]
  {\renewcommand\theinnercustomthm{#1}\innercustomthm}
  {\endinnercustomthm}

\newcommand\ar[3]
{
  \!\!\ensuremath
  {
    \begin{tikzcd}[ampersand replacement=\&]
      #1\colon #2 \ar[r] \& #3
    \end{tikzcd}
  }\!\!
}

\newcommand\artxt[4]
{
  \!\!\ensuremath
  {
    \begin{tikzcd}[ampersand replacement=\&]
      #1\colon #2 \ar[r,"#4"] \& #3
    \end{tikzcd}
  }\!\!
}

\newcommand\isom{\mathrel{\stackon[-0.1ex]{\makebox*{\scalebox{1.08}{\AC}}{=\hfill\llap{=}}}{{\AC}}}}

\newcommand\visom{\rotatebox[origin=cc] {90} {$ \isom $}}

\DeclareFontFamily{U}{BOONDOX-cal}{\skewchar\font=45 }
\DeclareFontShape{U}{BOONDOX-cal}{m}{n}{
  <-> s*[1.05] BOONDOX-r-cal}{}
\DeclareFontShape{U}{BOONDOX-cal}{b}{n}{
  <-> s*[1.05] BOONDOX-b-cal}{}
\DeclareMathAlphabet{\mathcalb}{U}{BOONDOX-cal}{m}{n}
\SetMathAlphabet{\mathcalb}{bold}{U}{BOONDOX-cal}{b}{n}
\DeclareMathAlphabet{\mathbcalb}{U}{BOONDOX-cal}{b}{n}

\title{Lambek Invariants of Commutative Squares \\ in a Homological Category}
\thanks{The work of the first author was carried out in the framework of the State Task to the Sobolev Institute of Mathematics (Project FWNF--2022--0006). \\ \today }

\author{Yaroslav Kopylov}
\address{Yaroslav Kopylov
\newline\hphantom{iii} Sobolev Institute of Mathematics
\newline\hphantom{iii} 4 Koptyug Ave.
\newline\hphantom{iii} 630090, Novosibirsk, Russia}
\email{\href{mailto:yakop@math.nsc.ru}{yakop@math.nsc.ru}}

\author{Vadim Leshkov}
\address{Vadim Leshkov
\newline\hphantom{iii} Novosibirsk State University
\newline\hphantom{iii} 1 Pirogova St.
\newline\hphantom{iii} 630090, Novosibirsk, Russia}
\email{\href{mailto:v.leshkov@g.nsu.ru}{v.leshkov@g.nsu.ru}}

\begin{document}

\begin{abstract}
We consider the Lambek invariants (introduced by~Joachim Lambek in~1964) in the context of semiexact and homological categories in the sense of Grandis.
We generalize the Lambek isomorphism theorem to semiexact and homological categories. Throughout the text, for completeness we give proofs of some lemmas which are known in the additive case for the case of semiexact and homological categories.

\vspace{2mm}

\vspace{2mm}

\end{abstract}
\maketitle

\section{Introduction}\label{intro}

In 1964, in~\protect\cite{Lambek-1964}, Lambek proved the following assertion:
\smallskip

{\it Given a~commutative diagram}
$$
\begin{tikzcd}
    A \ar[r,"f"] \ar[d,"a"'] \ar[r, phantom, shift right=4ex, "S" marking] & B \ar[r,"g"] \ar[d,"b" description] \ar[r, phantom, shift right=4ex, "T" marking] & C \ar[d,"c"]\\
    A' \ar[r,"f'"'] & B' \ar[r,"g'"'] & C'
\end{tikzcd}
$$
{\it of groups and group homomorphisms with exact rows, there is a~natural
isomorphism
$$
({\rm Im}\, b\cap{\rm Im}\, f')/{\rm Im}\,(bf)
\cong
{\rm Ker}(cg)/({\rm Ker}\, b\cdot{\rm Ker}\, g).
$$ }

For a~commutative
square
$$
\begin{tikzcd}
    A \ar[r,"f"] \ar[d,"a"'] \ar[r, phantom, shift right=4ex, "S" marking] & B \ar[d,"b"] \\
    C \ar[r,"g"'] & D
\end{tikzcd}
$$
in the category of groups we put:
$$
{\rm Im}\,S:=({\rm Im}\, b\cap{\rm Im}\, g)/{\rm Im}\,(bf), \quad
{\rm Ker}\,S:={\rm Ker}\,(bf)/({\rm Ker}\, a\cdot{\rm Ker}\, f).
$$
Since ${\rm Ker}\, a$ and ${\rm Ker}\, f$ are normal subgroups in~${\rm Ker}(bf)$,
their product ${\rm Ker} a\cdot{\rm Ker} f$ is the~normal subgroup generated by~them, and hence the quotient group ${\rm Ker}\, S$ is always defined, whereas ${\rm Im}(bf)$ is not always a normal subgroup in ${\rm Im}\, b \cap {\rm Im}\, g$.
However, in Lambek's theorem ${\rm Im}(bf) = b ({\rm Ker}\, g) \trianglelefteq ({\rm Im}\, b \cap {\rm Im}\, f')$.

Later Leicht~\protect\cite{Le} extended the Lambek's theorem to a~wider class of~categories that
includes the~category of~groups~$\mathcalb{G\!r\!p}$, thus giving a~homological proof.
In~\protect\cite{No1}, Nomura in~particular proved Lambek's isomorphism 
in~a~Puppe exact category.

In~\cite{Ko05_2}, these invariants were considered in~a~quasi-abelian category, and
Lambek's isomorphism was established provided that the~morphism~$b$
is exact.

In~the~present article, we prove Lambek's isomorphism in~the~framework of~Grandis homological
categories (see~\cite{Gr92,MGrandis-2013}). These categories were introduced as a~tool for~developing
homological algebra ``in~a~strogly non-abelian setting.'' In~recent years, Grandis homological
categories have found applications in~homological algebra and the theory of approximate representations of groups~\cite{ConnCons2019,Fr2023}.

The structure of the paper is as follows. In Section~{\ref{sec-null}}, we recall the~definitions
of~various classes of~categories with null morphisms and prove some auxiliary assertions.
In Section~{\ref{sec-lamb}}, we introduce the~canonical Lambek morphism in~a~semiexact category
and prove that this is an~isomorphism in~a~Grandis homological category.

\section{Categories with null morphisms}\label{sec-null}

\subsection{Semiexact categories}

First, recall the notion of an ideal of a category and the notions of kernel and cokernel with respect to an ideal introduced by Grandis in \cite[1.3.1]{MGrandis-2013}.

Let $\mathcalb{C}$ be a category.
Consider a class of morphisms $\mathcalb{N} \subseteq {\rm Mor}(\mathcalb{C})$ in $\mathcalb{C}$.
The class $\mathcalb{N}$ is called an \emph{ideal} in $\mathcalb{C}$ if for every morphism $f \in \mathcalb{N}$ any legitimate composite $hfg$ belongs to $\mathcalb{N}$.
The notion of an ideal of a category coincides with the usual notion of an ideal of a semigroup if we consider the class ${\rm Mor}(\mathcalb{C})$ as a semigroup with respect to a partially defined associative operation -- composition.
A pair $(\mathcalb{C}, \mathcalb{N})$ of a category $\mathcalb{C}$ equipped with an ideal $\mathcalb{N}$ is called a \emph{category with null morphisms} and morphisms from $\mathcalb{N}$ are called \emph{null morphisms}.
We denote the fact that a morphism $f\colon A \to B$ is null by the following label over the arrow
$$
\artxt{f}{A}{B}{{\rm null}}
\ \overset{{\rm def}}{\Longleftrightarrow}\
f \in \mathcalb{N}.$$
An object $A \in {\rm Ob}(\mathcalb{C})$ is called a \emph{null object} if ${\rm id}_A \in \mathcalb{N}$.
An ideal $\mathcalb{N}$ is \emph{closed} if there exists a class $\mathcalb{O} \subseteq {\rm Ob}(\mathcalb{C})$ such that every morphism from $\mathcalb{N}$ factors through some object in $\mathcalb{O}$.

In a category with null morphisms $(\mathcalb{C}, \mathcalb{N})$, a~morphism
$k\colon K \to A$ is called a \emph{kernel} of a morphism $f\colon A \to B$ if $fk\in \mathcalb{N}$ and for any morphism $x$ with $fx\in \mathcalb{N}$ there exists a~unique morphism $x'$ such that $x=kx'$. Dually, a~morphism
$l\colon B \to L$ is called a \emph{cokernel} of a morphism $f\colon A \to B$ if $lf\in \mathcalb{N}$ and for any morphism $y$ with $yf\in \mathcalb{N}$ there exists a~unique morphism $y'$ such that $y=y'l$.
If $k\colon K \to A$ is a~kernel of $f\colon A \to B$ then we write $K = {\rm Ker}\, f$, $k = {\rm ker}\, f$; if $l\colon B \to L$ is a~cokernel of~$f$ then we write $L={\rm Coker}f$ and
$l={\rm coker}f$. All we have said is schematically depicted in~the~diagram
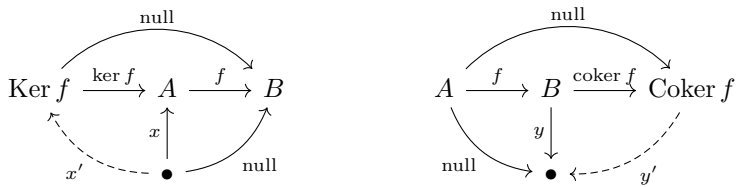
\[
\begin{tikzcd}
    {\rm Ker}\,f \ar[r,"{\rm ker}\, f"] \ar[rr,bend left=45,"{\rm null}"] & A \ar[r,"f"] & B & & A \ar[r,"f"] \ar[rr,"{\rm null}", bend left=45] \ar[dr,"{\rm null}"',bend right=30] & B \ar[r,"{\rm coker}\, f"] \ar[d,"y"'] & {\rm Coker}\, f \ar[dl,dashed,bend left=30,"y'"] 
  \\
   & \bullet \ar[ul,dashed,bend left=30,"x'"] \ar[u,"x"] \ar[ur,"{\rm null}"',bend right=30] & & & & \bullet
\end{tikzcd}
\].
    
The \emph{image} and \emph{coimage} of a morphism $f\colon A \to B$ are defined as 
$$\ar{{\rm im}\, f := {\rm ker} ({\rm coker}\, f)}{{\rm Im}\, f}{B}$$
and
$$\ar{{\rm coim}\, f := {\rm coker} ({\rm ker}\, f)}{A}{{\rm Coim}\, f},$$
respectively.
We say that a~morphism $f$ is an~$\mathcalb{N}$\emph{-monomorphism} (respectively, $\mathcalb{N}$\emph{-epimorphism}) if it satisfies the left-hand condition (respectively, the right-hand condition):
$$
\text{if}\ fh\ \text{is null then}\ h\ \text{is null};
\quad\quad
\text{if}\ kf\ \text{is null then}\ k\ \text{is null}.
$$
It is clear that ${\rm ker}\, f$ is an~$\mathcalb{N}$-monomorphism and ${\rm coker}\, f$ is an $\mathcalb{N}$-epimorphism.
The arrows $\rightarrowtail$ and $\twoheadrightarrow$ are reserved for kernels and cokernels respectively.
Note that if $(f\colon X \to Y) \in \mathcalb{N}$ then ${\rm ker}\, f = {\rm id}_X$ and ${\rm coker}\, f = {\rm id}_Y$.
We make use of the following important lemma (see \cite[1.3.1]{MGrandis-2013}):

\begin{lemma}\label{lemma_n_mono_ker_null_and_dual}
    The following are equivalent:
    \begin{itemize}
        \item[{\rm (i)}]
        $f$ is an $\mathcalb{N}$-monomorphism;
        \item[{\rm (ii)}]
        ${\rm ker}\, f \in \mathcalb{N}$;
        \item[{\rm (iii)}]
        ${\rm ker}(fh) = {\rm ker}\, h$ for any $h$.
    \end{itemize}
    Dually, the following are equivalent:
    \begin{itemize}
        \item[{\rm (i)}]
        $f$ is an $\mathcalb{N}$-epimorphism;
        \item[{\rm (ii)}]
        ${\rm coker}\, f \in \mathcalb{N}$;
        \item[{\rm (iii)}]
        ${\rm coker}(kf) = \mathrm{coker}\, k$ for any $k$.
    \end{itemize}
\end{lemma}

Let $(\mathcalb{C}, \mathcalb{N})$ be a category with null morphisms.
Consider the~following two axioms~\cite{MGrandis-2013}.

\begin{axiom}\label{first_axiom}
    The ideal $\mathcalb{N}$ is closed.
\end{axiom}

\begin{axiom}\label{second_axiom}
    Every morphism has a kernel and cokernel with respect to $\mathcalb{N}$.
\end{axiom}

A category with null morphisms is called \emph{semiexact} if it satisfies Axioms \ref{first_axiom} and \ref{second_axiom}.
Semiexact categories are also called \emph{ex1-categories}.

Consider a morphism $f\colon A \to B$ in a semiexact category $(\mathcalb{C}, \mathcalb{N})$.
Note that $({\rm coker}\, f) f$ is null.
Hence, there exists a unique morphism $f'\colon A \to {\rm Im}\, f$ such that $f = ({\rm im}\, f) f'$.
Note that $f' ({\rm ker}\, f)$ is null.
Hence, $f'$ factors through the coimage of f, i.e. there exists a morphism $\overline{f}\colon {\rm Coim}\, f \to {\rm Im}\, f$.
\[
\begin{tikzcd}
    {\rm Ker}\, f \ar[r,"{\rm ker}\, f",tail] & A \ar[r,"f"] \ar[d,"{\rm coim}\, f"',two heads] \ar[rd,"f'"',dashed,bend left=25] & B \ar[r,"{\rm coker}\, f",two heads] & {\rm Coker}\, f\\
    & {\rm Coim}\, f \ar[r,dashed,"\overline{f}"'] & {\rm Im}\, f \ar[u,"{\rm im}\, f"',tail]
\end{tikzcd}
\]
Therefore, any morphism $f$ in a semiexact category can be factored as 
\begin{equation}\label{normal-fact}
f = ({\rm im}\, f)\, \overline{f}\, ({\rm coim}\, f). 
\end{equation}
We refer to~\eqref{normal-fact} as a~normal (or canonical) decomposition
of~$f$.
Two normal decompositions of~$f$ are canonically isomorphic.
The~morphism~$\overline{f}$ in~\eqref{normal-fact}
is defined uniquely once ${\rm im}\, f$ and ${\rm coim}\, f$ are chosen.
If $\overline{f}$ is an isomorphism, then the morphism $f$ is called \emph{exact}.
Usually, when $f$ is exact, we omit the middle isomorphism and write $f = ({\rm im}\, f)\, ({\rm coim}\, f)$.
Omitting the middle isomorphism is appropriate because it is always possible to choose the correct isomorphic representatives of kernels and cokernels to make $\overline{f}$ equal to the indentity morphism.

We prove several lemmas which are useful for the results of the present work.

The following lemma is an analog of \cite[Lemma 2.2]{KoWe2012}.

\begin{lemma}[]\label{kop_weg_lemma}
Given morphisms  $f\colon X \to Y$ and $g\colon Y \to Z$  in a semiexact category $(\mathcalb{C}, \mathcalb{N})$, the following identities hold:
$$\mathrm{coker}(g\, (\mathrm{im}\, f)) = \mathrm{coker}(gf),\quad
\mathrm{ker}((\mathrm{coim}\, g)\,  f) = \mathrm{ker}(gf).$$
\end{lemma}
\begin{proof}
    Observe that ${\rm coker}(g ({\rm im}\, f)) g f \in \mathcalb{N}$.
    Given a morphism $x\colon Z \to W$ with $x g f \in \mathcalb{N}$, there exists a unique morphism $y\colon {\rm Coker}\, f \to W$ such that $x g = y ({\rm coker}\ f)$.
    Hence, $x g ({\rm im} f) = y ({\rm coker}\ f) ({\rm im} f) \in \mathcalb{N}$ and there exists a unique $z\colon {\rm Coker}(g ({\rm im}\, f)) \to W$ such that $x = z ({\rm coker}(g ({\rm im}\, f)))$.
    Therefore, ${\rm coker}(g ({\rm im}\, f))$ is a cokernel of $gf$.
    The second identity is obtained by duality.
\end{proof}

\begin{corollary}\label{yak_lemma_corr_1}
    In a semiexact category $(\mathcalb{C}, \mathcalb{N})$, given morphisms $f$ and $g$ for~which $gf$ is defined, we have:
    \begin{itemize}
        \item[\rm (i)] if kernels are closed under composition and $g$ is a kernel then $\mathrm{im}(gf) = g (\mathrm{im}\, f)$;
        \item[\rm (ii)] dually, if cokernels are closed under composition and $f$ is a cokernel then $\mathrm{coim}(gf) = (\mathrm{coim}\, g) f$.
    \end{itemize}
\end{corollary}
\begin{proof}
    (i) Observe that by Lemma \ref{kop_weg_lemma} we have 
    $${\rm im}(gf) = {\rm ker}({\rm coker}(gf)) = {\rm ker}\, {\rm coker}(g ({\rm ker} ({\rm coker}\, f))).$$
    Note that $g ({\rm ker} ({\rm coker}\, f))$ is a kernel.
    Hence, it is equal to its image.
    Assertion (ii) holds by duality.
\end{proof}

\begin{lemma}[Properties of exact morphisms]\label{ker_coker_same}
    The following assertions hold in a semiexact category $(\mathcalb{C},\mathcalb{N})$:
    \begin{enumerate}
        \item[\rm (i)] $\mathrm{ker}\, \alpha$ is an exact $\mathcalb{N}$-monomorphism and $\mathrm{coker}\, \alpha$ is an exact $\mathcalb{N}$-epimorphism for any morphism~$\alpha$;
        \item[\rm (ii)] $\alpha = \mathrm{im}\, \alpha$ if and only if $\alpha$ is an exact $\mathcalb{N}$-monomorphism and $\alpha = \mathrm{coim}\, \alpha$ if and only if $\alpha$ is an exact $\mathcalb{N}$-epimorphism;
        \item[\rm (iii)]
        if $f\colon A \to B$ is a morphism such that $f = h g$, where $h$ is a kernel and $g$ is a cokernel, then $f$ is exact, i.e., $f = (\mathrm{im}\, f)\, (\mathrm{coim}\, f)$.
    \end{enumerate}
\end{lemma}
\begin{proof}
    \begin{enumerate}
        \item[(i)]
        It is obvious that kernels are $\mathcalb{N}$-monomorphisms and cokernels are $\mathcalb{N}$-epimorphisms.
        We show that (co)kernels are exact.
        Suppose that $\beta$ is a kernel: $\beta = {\rm ker}\, \alpha$ for some morphism $\alpha$.
        Consider its normal decomposition.
        Note that
        $$\beta = ({\rm im}\, \beta) \overline{\beta} ({\rm coim}\, \beta) = ({\rm im}\, \beta)\, \overline{\beta}\, \overbrace{({\rm coker}(\, \underbrace{{\rm ker}\, \beta}_{\in \mathcalb{N}}\, ))}^{{\rm id}},$$
        because $\beta$ is a kernel.
        Assume  without loss of generality that $\beta = {\rm ker}\, \alpha = ({\rm im}\, \beta) \overline{\beta}$.
        Hence, ${\rm im}\, \beta = {\rm ker}({\rm coker}\, \beta) = {\rm ker}({\rm coim}\, \alpha)$.
        By Corollary \ref{yak_lemma_corr_1}, ${\rm ker}({\rm coim}\, \alpha) = {\rm ker}\, \alpha$.
        Therefore, ${\rm im}\, \beta = {\rm ker}\, \alpha$.
        Hence, $\overline{\beta}$ is an isomorphism and $\beta$ is exact.
        Consequently, kernels are exact and, by duality, cokernels are also exact.
    
        \item[(ii)]
        If $\alpha = {\rm im}\, \alpha$ then $\alpha$ is an exact $\mathcalb{N}$-monomorphism by the definition of an exact morphism and the definition of a kernel.
        Suppose that $\alpha$ is an exact $\mathcalb{N}$-monomorphism.
        Hence, $\alpha = ({\rm im}\, \alpha) ({\rm coim}\, \alpha)$ and $\alpha$ has a null kernel: ${\rm ker}\, \alpha \in \mathcalb{N}$.
        Hence, ${\rm coim}\, \alpha = {\rm coker}({\rm ker}\, \alpha)$ is an isomorphism and $\alpha$ is its own image.
        Dually, $\alpha = {\rm coim}\, \alpha$ if and only if $\alpha$ is an exact $\mathcalb{N}$-epimorphism.

        \item[(iii)]
        Note that $f\, ({\rm ker}\, g) = h g\, ({\rm ker}\, g) \in \mathcalb{N}$.
        Hence, there exists a unique morphism $\alpha\colon {\rm Ker}\, g \to {\rm Ker}\, f$ such that ${\rm ker}\, g = ({\rm ker}\, f)\, \alpha$.
        Also, note that $h g\, ({\rm ker}\, f) = f\, ({\rm ker}\, f) \in \mathcalb{N}$ and $h$ is an~$\mathcalb{N}$-monomorphism.
        Hence, $g\, ({\rm ker}\, f) \in \mathcalb{N}$ and ${\rm ker}\, f$ factors through ${\rm ker}\, g$ uniquely, i.e. there exists a unique morphism $\beta\colon {\rm Ker}\, f \to {\rm Ker}\, g$ such that ${\rm ker}\, f = ({\rm ker}\, g)\, \beta$.
        By the universal property of the~kernel, we get $\alpha \beta = {\rm id}_{{\rm Ker}\, f}$ and $\beta \alpha = {\rm id}_{{\rm Ker}\, g}$.
        Hence, the~kernels of $f$ and $g$ are isomorphic and, dually, the~cokernels of $f$ and $h$ are isomorphic.
        \[
        \begin{tikzcd}
            {\rm Ker}\, f \ar[r,"{\rm ker}\, f",tail] \ar[d,"\beta",bend left=30,dashed] \ar[d,phantom,"\visom"] & A \ar[rr,"f"',bend left=-35] \ar[r,"g"] & \bullet \ar[r,"h"] & B \ar[r,"{\rm coker}\, f",two heads] \ar[rd,"{\rm coker}\, h"',tail,bend right=30] & {\rm Coker}\, f \ar[d,phantom,"\visom"]\\
            {\rm Ker}\, g \ar[u,"\alpha",dashed,bend left=30] \ar[ur,"{\rm ker}\, g"',tail,bend right=30] & & & & {\rm Coker}\, h
        \end{tikzcd}
        \]
        From the~isomorphism ${\rm ker}\, g = {\rm ker}\, f$ we get ${\rm coim}\, g = {\rm coim}\, f$.
        Hence, by Lemma \ref{ker_coker_same}, $g = {\rm coim}\, g = {\rm coim}\, f$.
        Dually, $h = {\rm im}\, h = {\rm im}\, f$.
    \end{enumerate}
\end{proof}

A sequence of morphisms
\begin{equation}\label{long_sequence}
\begin{tikzcd}
    \ldots \ar[r,"f_{-2}"] & X_{-1} \ar[r,"f_{-1}"] & X_0 \ar[r,"f_0"] & X_1 \ar[r,"f_1"] & X_2 \ar[r,"f_2"] & \ldots
\end{tikzcd}
\end{equation}
in a semiexact category is called a \emph{null-sequence} (or a sequence \emph{of order two}) if any composition of two consecutive morphisms $f_{i+1} f_i$ is null.
As in \cite[1.5.2]{MGrandis-2013}, we call sequence (\ref{long_sequence}) \emph{exact at $X_i$} if ${\rm im}\, f_{i-1} = {\rm ker}\, f_i$, or, equivalently, (by Lemma \ref{kop_weg_lemma}) ${\rm coker}\, f_{i-1} = {\rm coim}\, f_i$ due to the~equalities
$${\rm coker}\, f_{i-1} = {\rm coker}({\rm im}\, f_{i-1}) = {\rm coker}({\rm ker}\, f_{i}) = {\rm coim}\, f_i.$$

The~following two lemmas are formulated in~\cite{MGrandis-2013} and are not hard to~prove.

\begin{lemma}\label{ker_coker_pb_po}
In a semiexact category $(\mathcalb{C}, \mathcalb{N})$, if $\alpha\colon A \twoheadrightarrow B$ is a cokernel, i.e. $\alpha = \mathrm{coker}\, \alpha'$ for a morphism $\alpha'\colon \bullet \to A$, then for every $g\colon A \to A'$ there exists a morphism $g'\colon B \to \mathrm{Coker}(g \alpha')$ such that the left-hand square is a pushout:
    \[
    \begin{tikzcd}
        \bullet \ar[r,"\alpha'"] & A \ar[r,"\alpha\, =\, \mathrm{coker}\, \alpha'",two heads] \ar[d,"g"'] \ar[r, phantom, shift right=4ex, "{\rm PO}" marking] & B \ar[d,"g'",dashed]\\
        & A' \ar[r, "\mathrm{coker}(g \alpha')"',two heads] & \mathrm{Coker}(g \alpha')
    \end{tikzcd}
    \quad
    \quad
    \begin{tikzcd}
        \mathrm{Ker}(\beta' f) \ar[r,"\mathrm{ker}(\beta' f)",tail] \ar[d,dashed,"f'"'] & D' \ar[d,"f"]\\
        C \ar[r, phantom, shift left=4ex, "{\rm PB}" marking] \ar[r,"\beta\, =\, \mathrm{ker}\, \beta'"',tail] & D \ar[r,"\beta'"'] & \bullet
    \end{tikzcd}
    \]
    Dually, if $\beta\colon C \rightarrowtail D$ is a kernel, i.e. $\beta = \mathrm{ker}\, \beta'$ with $\beta'\colon D \to \bullet$, then for every $f\colon D' \to D$ there exists a morphism $f'\colon \mathrm{Ker}(\beta' f) \to C$ such that the right-hand square is a pullback.
\end{lemma}

The following property of the middle morphism $\overline{\alpha}$ in a~normal decomposition of a morphism $\alpha$ was proved for the case of semiexact categories in \cite{MGrandis-2013} with an additional requirement on compositions of kernels and cokernels.
It can be deduced using Lemma \ref{kop_weg_lemma}.

\begin{lemma}\label{ex2_middle_morphism_th}
    In a semiexact category $(\mathcalb{C},\mathcalb{N})$, if kernels (cokernels) are closed under composition then $\overline{\alpha}$ is an~$\mathcalb{N}$-epimorphism ($\mathcalb{N}$-monomorphism) for any morphism $\alpha$.
\end{lemma}

\subsection{Homological categories}

Let $(\mathcalb{C},\mathcalb{N})$ be a semiexact category.
Consider the~following axiom:

\begin{axiom}\label{third_axiom}
    Kernels and cokernels are closed under composition.
\end{axiom}

A semiexact category is called an~\emph{ex2-category}~\cite[2.1.3]{MGrandis-2013} if it satisfies Axiom~\ref{third_axiom}.

\begin{axiom}\label{fourth_axiom}
    Consider a kernel
    $i\colon A \rightarrowtail B$ and a cokernel
    $q\colon B \twoheadrightarrow C$
    such that $(\mathrm{coker}\, i)(\mathrm{ker}\, q) \in \mathcalb{N}$.
    Then there exist a cokernel
    $\pi\colon A \twoheadrightarrow H$
    and a kernel
    $\iota\colon H \rightarrowtail C$
    which make the following square commute:
    \[
    \begin{tikzcd}
        A \ar[d,two heads,dashed,"\pi"'] \ar[r,tail,"i"] & B \ar[d,two heads,"q"]\\
        H \ar[r,tail,dashed,"\iota"'] & C
    \end{tikzcd}
    \]
    The object $H$ in Axiom~\ref{fourth_axiom} is called a \emph{homology object}.
\end{axiom}

An ex2-category is called \emph{homological} (or \emph{ex3-category}) if it satisfies Axiom~\ref{fourth_axiom}.
It is proved in~\cite[2.2.2]{MGrandis-2013} that for an ex2-category, Axiom~\ref{fourth_axiom} is equivalent to the following axiom:

\begin{customaxiom}{4*}\label{po_pb_ker_coker}
    The pullback of a pair $\bullet \twoheadrightarrow \bullet \leftarrowtail \bullet$ is of the form $\bullet \leftarrowtail \bullet \twoheadrightarrow \bullet$.
    Dually, the pushout of a pair $\bullet \leftarrowtail \bullet \twoheadrightarrow \bullet$ is of the form $\bullet \twoheadrightarrow \bullet \leftarrowtail \bullet$.
\end{customaxiom}

It is important to note that in the category of groups $\mathcalb{G\!r\!p}$ is semiexact and it is not homological and not even an~ex2-category.
In $\mathcalb{G\!r\!p}$ the pullback of a pair $\bullet \twoheadrightarrow \bullet \leftarrowtail \bullet$ exists and has the form $\bullet \leftarrowtail \bullet \twoheadrightarrow \bullet$\ . However, the composition of~two kernels
need not be a~kernel in~$\mathcalb{G\!r\!p}$.

\begin{lemma}\label{lemma_pullback_of_exact_is_exact}
    Suppose that in a homological category $(\mathcalb{C},\mathcalb{N})$, the~commutative square
    \begin{equation}\label{comm_square}
    \begin{tikzcd}
        A \ar[r,"\alpha"] \ar[d,"g"']& B \ar[d,"f"]\\
        C \ar[r,"\beta"'] & D
    \end{tikzcd}
    \end{equation}
    is a pullback, $f$ is a kernel, and $\beta$ is exact, then $\alpha$ is exact.  Dually, if {\rm(\ref{comm_square})} is a pushout, $g$ is a cokernel, and $\alpha$ is exact, then $\beta$ is exact.
\end{lemma}
\begin{proof}
    Suppose that (\ref{comm_square}) is a~pullback, $\beta$ is exact, and $f$ is a~kernel.  We may assume that $\beta = ({\rm im}\, \beta) ({\rm coim}\, \beta)$.  Consider the~diagram
    \begin{equation}\label{pullb-decomp}
    \begin{tikzcd}
        A \ar[r,"\alpha_0",dashed] \ar[d,"g"'] \ar[rr,bend left=45,"\alpha"] & \bullet \ar[r,"\alpha_1"] \ar[d,"h"'] \ar[r, phantom, shift right=4ex, "{\rm PB}" marking] & B \ar[d,"f"]\\
        C \ar[r,"{\rm coim}\, \beta"'{xshift=-3pt}] & \bullet \ar[r,"{\rm im}\, \beta"'] & D
    \end{tikzcd}
    \end{equation}
    where the~right-hand square is a~pullback and $\alpha_0$ is the unique morphism with $\alpha=\alpha_1\alpha_0$ and $({\rm coim}\beta)g=h\alpha_0$.
    We show that the~left-hand square in~(\ref{pullb-decomp}) is also a~pullback. Consider two coinitial morphisms $x$ and $y$ such that $hy = ({\rm coim}\, \beta) x$:
    \[
    \begin{tikzcd}
        \bullet \ar[rrd,"y",bend left=30] \ar[rdd,"x"',bend right=30] \ar[dashed,rd,"z"] & & \\
        & A \ar[r,"\alpha_0"] \ar[d,"g"'] \ar[r, phantom, shift right=4ex, "{\rm PB}" marking] & \bullet \ar[d,"h"'] \ar[r,"\alpha_1"] \ar[r, phantom, shift right=4ex, "{\rm PB}" marking] & B \ar[d,"f"] \\
        & C \ar[r,"{\rm coim}\, \beta"'] \ar[rr,"\beta"',bend right=45] & \bullet \ar[r,"{\rm im}\, \beta"'] & D
    \end{tikzcd}
    \]
    Therefore, $\beta x = ({\rm im}\, \beta) ({\rm coim}\, \beta) x = ({\rm im}\, \beta) hy = f \alpha_1 y$. Since (\ref{comm_square}) is a~pullback, there exists a unique $z$ such that $\alpha_1 y = \alpha z$ and $x = gz$. From the identities $\alpha_1 y = \alpha z = \alpha_1 \alpha_0 z$
    we have $y = \alpha_0 z$. Since the left-hand square in~(\ref{pullb-decomp}) is also a pullback, by Axiom \ref{po_pb_ker_coker}, $\alpha_0$ is a cokernel. Therefore, $\alpha = \alpha_1 \alpha_0$, where $\alpha_0$ is a~cokernel and $\alpha_1$ is a kernel.
    Thus, by Lemma~\ref{ker_coker_same}(iii), $\alpha_1 = {\rm im}\, \alpha$, $\alpha_0 = {\rm coim}\, \alpha$ and $\alpha_1$ is exact.

The~second assertion of~the~lemma is obtained by~duality.    
\end{proof}

The following property is equivalent to~right (left) semiabelianity for~preabelian categories (see~\cite{KoWe2012}). It also holds in~an~ex2-category.

\begin{lemma}\label{ker-cancelation}
    Let $(\mathcalb{C},\mathcalb{N})$ be a semiexact category.
    If kernels in $\mathcalb{C}$ are stable under composition and $\alpha \beta$ is a kernel then $\beta$ is a kernel. Dually, If cokernels in $\mathcalb{C}$ are stable under composition and  $\alpha \beta$ is a cokernel then $\alpha$ is a cokernel.
\end{lemma}

\begin{proof}
    Let $\alpha \beta$ be a kernel.
    Present~$\beta$ as the composition $\beta = i p$, where $i = {\rm im}\, \beta$ is a kernel and $p = \overline{\beta}({\rm coim}\, \beta)$ is an~$\mathcalb{N}$-epimorphism.
    By Lemma~\ref{ex2_middle_morphism_th}, $\overline{\beta}$ is an~$\mathcalb{N}$-epimorphism. Therefore,
    $\alpha \beta = {\rm im}(\alpha \beta) = {\rm im}(\alpha i)$.
    Note that $\alpha \beta = \alpha i p$ and $\alpha i = \alpha \beta q$, where $q = \overline{\alpha i}\ {\rm coim}(\alpha i)$.
    Thus, $\alpha i = \alpha i p q$ and $p q = {\rm id}$ since $\alpha i$ is an~$\mathcalb{N}$-monomorphism.
    Note that $\alpha i p = \alpha \beta$ is an~$\mathcalb{N}$-monomorphism, and hence $p$ is an~$\mathcalb{N}$-monomorphism.
    Therefore, $p q p = p$ and $qp = {\rm id}$, because $p$ is an~$\mathcalb{N}$-monomorphism.
    We conclude that $p$ is an isomorphism, which means that $\beta$ is a kernel.
    The second assertion is obtained by duality.
\end{proof}

Consider a commutative square (\ref{comm_square}) in a semiexact category.
Note that
$\beta g ({\rm ker}\, \alpha) = f \alpha ({\rm ker}\, \alpha) \in \mathcalb{N}$,
and so there exists a unique morphism $\widehat{g}\colon {\rm Ker}\, \alpha \to {\rm Ker}\, \beta$ such that $(\ker\beta)\widehat{g}=g\ker\alpha$, which we call \emph{the~morphism~of~kernels~induced~by~$g$}. 
Dually, there exists a unique morphism of cokernels $\widehat{f}\colon {\rm Coker}\, \alpha \to {\rm Coker}\, \beta$ such that
$\widehat{f}{\rm coker}\,\alpha= ({\rm coker}\,\beta)f$, which we call \emph{the~morphism~of~cokernels~induced~by~$f$}.
\[
    \begin{tikzcd}
        {\rm Ker}\, \alpha \ar[r,"{\rm ker}\, \alpha",tail] \ar[d,dashed,"\widehat{g}"'] & A \ar[r,"\alpha"] \ar[d,"g"'] & B \ar[r,"{\rm coker}\, \alpha",two heads] \ar[d,"f"] & {\rm Coker}\, \alpha \ar[d,"\widehat{f}",dashed] \\
        {\rm Ker}\, \beta \ar[r,"{\rm ker}\, \beta"',tail] & C \ar[r,"\beta"'] & D \ar[r,"{\rm coker}\, \beta"',two heads] & {\rm Coker}\, \beta
    \end{tikzcd}
\]

\begin{lemma}\label{lemma_morphism_of_kernel_cokernels}
    If {\rm \eqref{comm_square}} is a pullback in a homological category $(\mathcalb{C},\mathcalb{N})$, $\beta$ is exact, and $f$ is a kernel then the morphism of cokernels $\widehat{f}$ induced by $f$ is an~$\mathcalb{N}$-monomorphism.
    Dually, if {\rm \eqref{comm_square}} is a pushout, the morphism $\alpha$ is exact, and $g$ is cokernel then the morphism of kernels $\widehat{g}$ induced by $g$ is an~$\mathcalb{N}$-epimorphism.
\end{lemma}

\begin{proof}
    Suppose that the square is a pullback and $\beta$ is exact, i.e. $\beta = ({\rm im}\, \beta) ({\rm coim}\, \beta)$.
    Hence, by Lemma \ref{lemma_pullback_of_exact_is_exact} the morphism $\alpha$ is exact with $\alpha = ({\rm im}\, \alpha) ({\rm coim}\, \alpha)$.
    Let $x\colon X \to {\rm Coker}\, \alpha$ be a morphism such that $\widehat{f} x \in \mathcalb{N}$, hence, $\widehat{f}({\rm im}\, x) \in \mathcalb{N}$.
    Consider the pullback 
    \[
    \begin{tikzcd}
        P \ar[r,two heads,"\gamma"] \ar[d,tail,"y"'] \ar[r, phantom, shift right=4ex, "{\rm PB}" marking] & {\rm Im}\, x \ar[d,tail,"{\rm im}\, x"]\\
        B \ar[r,two heads,"{\rm coker}\, \alpha"'] & {\rm Coker}\, \alpha
    \end{tikzcd}
    \]
     (which exists due to Lemma~\ref{ker_coker_pb_po}).
    Note that $({\rm coker}\, \beta) f y = \widehat{f} ({\rm im}\, x) \gamma \in \mathcalb{N}$.
    Hence, there exists a unique morphism $z\colon P \to {\rm Im}\, \beta$ such that $({\rm im}\, \beta) z = f y$.
    Consider the pullback of the pair $(f, {\rm im}\, \beta)$.
    Note that the pullback of ${\rm im}\, \beta$ along~$f$ is ${\rm im}\, \alpha$ (see the proof of Lemma \ref{lemma_pullback_of_exact_is_exact}).
    We denote the pullback of $f$ along ${\rm im}\, \beta$ by $h$.
    By the universal property of a pullback there exists a unique morphism $t\colon P \to {\rm Im}\, \alpha$ such that $z = h t$ and $y = ({\rm im}\, \alpha) t$.
    Note that $({\rm im}\, x) \gamma = ({\rm coker}\, \alpha) ({\rm im}\, \alpha) t \in \mathcalb{N}$ and $\gamma$ is an~$\mathcalb{N}$-epimorphism by Axiom~\ref{po_pb_ker_coker}.
    Hence, ${\rm im}\, x \in \mathcalb{N}$ and $x \in \mathcalb{N}$.
    Therefore, $\widehat{f}$ is an~$\mathcalb{N}$-monomorphism.
    The diagram for the proof looks as follows:
    \[
    \begin{tikzcd}
        & & X \ar[d,"\overline{x} ({\rm coim}\, x)"' near start] \ar["x",dd,bend left=75] \\
        & P \ar[r,"\gamma",two heads] \ar[d,"y"',tail] \ar[r, phantom, shift right=4ex, "{\rm PB}" marking] \ar["z"',ldd,dashed,bend right=90] \ar[ld,"t",dashed,bend right=45] & {\rm Im}\,X \ar[d,"{\rm im}\, x",tail] \\
        {\rm Im}\, \alpha \ar[r,"{\rm im}\, \alpha",tail] \ar[d,"h"] \ar[r, phantom, shift right=4ex, "{\rm PB}" marking] & \bullet \ar[r,"{\rm coker}\, \alpha"',two heads] \ar[d,"f"] & {\rm Coker}\, \alpha \ar[d,"\widehat{f}"] \\
        {\rm Im}\, \beta \ar[r,"{\rm im}\, \beta"',tail] & \bullet \ar[r,"{\rm coker}\, \beta"',two heads] & {\rm Coker}\, \beta
    \end{tikzcd}
    \]
    The second assertion can be obtained by duality.
\end{proof}

\begin{lemma}[The Composition Lemma]\label{composition_lemma}
    In a homological category $(\mathcalb{C},\mathcalb{N})$, suppose that the composition $gf$ of two morphisms $f$ and $g$ is defined.
    Then there exists a null-sequence    \begin{equation}\label{composition_lemma_sequence}
    \begin{tikzcd}
        \bullet \ar[r,"{\rm null}"] & \mathrm{Ker}\, f \ar[r,"\varphi"] & \mathrm{Ker}(gf) \ar[r,"\psi"] \ar[d, phantom, ""{coordinate, name=Z}] & {\rm Ker}\, g \ar[
            dll,
            "\chi"',
            rounded corners,
            to path={ -- ([xshift=2ex]\tikztostart.east)
            |- (Z) [near end]\tikztonodes
            -| ([xshift=-2ex]\tikztotarget.west) 
            -- (\tikztotarget)}
        ] & \\
        & \mathrm{Coker}\, f \ar[r,"\varepsilon"] & \mathrm{Coker}(gf) \ar[r,"\omega"] & \mathrm{Coker}\, g \ar[r,"{\rm null}"] & \bullet
    \end{tikzcd}
    \end{equation}
    which is exact at $\mathrm{Ker}\, f$, $\mathrm{Ker}(gf)$, $\mathrm{Coker}(gf)$, $\mathrm{Coker}\, g$.
    Moreover,
    \begin{enumerate}
        \item
        $\varphi$ and $\omega$ are exact morphisms;
        \item
        if $f$ is exact morphism then {\rm \eqref{composition_lemma_sequence}} is exact at $\mathrm{Ker}\, g$ and $\psi$ is exact morphism.
        Dually, if $g$ is an exact morphism then {\rm \eqref{composition_lemma_sequence}} is exact at $\mathrm{Coker}\, f$ and $\varepsilon$ is exact morphism.
    \end{enumerate}
\end{lemma}

\begin{proof}
    We define $\varphi$, $\psi$, $\varepsilon$ and $\omega$ by the following equalities that follow from the universal property of a (co)kernel:
    ${\rm ker}f = ({\rm ker}(gf)) \varphi$,\ $f ({\rm ker}(gf)) = ({\rm ker}\, g) \psi$,\ $({\rm coker}(gf)) g = \varepsilon ({\rm coker}\, f)$,\ ${\rm coker}\, g = \omega ({\rm coker}(gf))$.
    We define the middle morphism $\chi$ in the diagram as $\chi := ({\rm coker}\, f) ({\rm ker}\, g)$.
    By Lemma~\ref{ker_coker_pb_po}, the second equality $f ({\rm ker}(gf)) = ({\rm ker}\, g) \psi$ forms a pullback.
    Dually, the third equality $({\rm coker}(gf)) g = \varepsilon ({\rm coker}\, f)$ forms a pushout.
    \[
    \begin{tikzcd}
        & & {\rm Coker}\, f \ar["\varepsilon",r,dashed] & {\rm Coker}(gf) \ar["\omega",rd,dashed,bend left=25]\\
        {\rm Ker}\, f \ar["{\rm ker}\, f",r,tail] \ar[rd,"\varphi"',dashed,bend right=25] & \bullet \ar["f",r] \ar[r, phantom, shift right=4ex, "{\rm PB}" marking] & \bullet \ar["g"',r] \ar["{\rm coker}\, f",u,two heads] \ar[r, phantom, shift left=4ex, "{\rm PO}" marking] & \bullet \ar["{\rm coker}\, g"',r,two heads] \ar["{\rm coker}(gf)"',u,two heads] & {\rm Coker}\, g\\
        & {\rm Ker}(gf) \ar["{\rm ker}(gf)",u,tail] \ar["\psi"',r,dashed] & {\rm Ker}\, g \ar["{\rm ker}\, g"',u,tail]
    \end{tikzcd}
    \]
    Since kernels (cokernels) are $\mathcalb{N}$-monomorphisms ($\mathcalb{N}$-epimorphisms), $\psi \varphi$, $\omega \varepsilon \in \mathcalb{N}$.
    Note that
    $$\chi \psi = ({\rm coker}\, f) ({\rm ker}\, g) \psi = ({\rm coker}\, f) f ({\rm ker}(gf)) \in \mathcalb{N}.$$
    Dually, $\varepsilon \chi \in \mathcalb{N}$.
    Hence, the morphisms $\varphi$, $\psi$, $\chi$, $\varepsilon$, $\omega$ form a null-sequence.

    It is obvious that the sequence is exact at ${\rm Ker}\, f$ and ${\rm Coker}\, g$.
    Observe that ${\rm ker}\, f = ({\rm ker}(gf)) \varphi$.
    Hence, by Lemma \ref{ker-cancelation}, $\varphi$ is a kernel.
    Suppose that $x\colon X \to {\rm Ker}(gf)$ is a morphism such that $\psi x \in \mathcalb{N}$.
    Hence, $({\rm ker}\, g) \psi x = f ({\rm ker}(gf)) x \in \mathcalb{N}$ and there exists a unique $t\colon X \to {\rm Ker}\, f$ such that $({\rm ker}(gf)) x = ({\rm ker}\, f) t$.
    Therefore, $({\rm ker}(gf)) x = ({\rm ker}(gf)) \varphi t$, $x = \varphi t$ and $\varphi = {\rm ker}\, \psi$.
    Now observe that ${\rm im}\, \varphi = \varphi = {\rm ker}\, \psi$, because $\varphi$ is a kernel.
    Hence, the sequence is exact at ${\rm Ker}(gf)$.
    Dually, it is exact at ${\rm Coker}(gf)$.

    Suppose that $f$ is exact.
    Hence, its pullback $\psi$ is exact and $\psi = ({\rm im}\, \psi) ({\rm coim}\, \psi)$.
    Consider a morphism $l\colon {\rm Coker}\, \psi \to {\rm Coker}\, f$ of cokernels in the following square induced by ${\rm ker}\, g$:
    \[
    \begin{tikzcd}
        \bullet \ar[r,"f"] & \bullet \ar[r,"{\rm coker}\, f",two heads] & {\rm Coker}\, f\\
        \bullet \ar[u,"{\rm ker}(gf)",tail] \ar[r,"\psi"'] & \bullet \ar[u,"{\rm ker}\, g",tail] \ar[ur,"\chi"] \ar[r,"{\rm coker}\, \psi"',two heads] & {\rm Coker}\, \psi \ar[u,"l"']
    \end{tikzcd}
    \]
    Note that, by Lemma~\ref{lemma_morphism_of_kernel_cokernels}, the morphism $l$ is an~$\mathcalb{N}$-monomorphism.
    By commutativity we have $\chi = l ({\rm coker}\, \psi)$, hence, ${\rm ker}\, \chi = {\rm ker}(l ({\rm coker}\, \psi)) = {\rm im}\, \psi$.
    Therefore, (\ref{composition_lemma_sequence}) is exact at ${\rm Ker}\, g$.
    Dually, if $g$ is exact then $\varepsilon$ is exact and (\ref{composition_lemma_sequence}) is exact at ${\rm Coker}\, f$.
\end{proof}

\section{Lambek's theorem in categories with null morphisms}\label{sec-lamb}

In a semiexact category $(\mathcalb{C}, \mathcalb{N})$, consider the square denoted by $S$:
\begin{equation}
\begin{tikzcd}
    A \ar[r,"f"] \ar[d,"a"'] \ar[r, phantom, shift right=4ex, "S" marking] & B \ar[d,"b"] \\
    C \ar[r,"g"'] & D
\end{tikzcd}
\end{equation}
In 1964 J.~Lambek \cite{Lambek-1964} defined the invariants ${\rm Im}\, S$ and ${\rm Ker}\, S$ of the commutative square $S$ in the category of groups.
We generalize these invariants for the case of categories with null morphisms as ${\rm Im}\, S := {\rm Coker}\, \lambda_S$, ${\rm Ker}\, S := {\rm Ker}\, \rho_S$, where $\lambda_S$ and $\rho_S$ are the morphisms defined by the universal property of a~pullback and a~pushout:
\[
\begin{tikzcd}
    & A \ar[dd,"a"'] \ar[rr,"f"] \ar[rd,dashed,"\lambda_S"] & & B \ar[d]\\
    {\rm Im}\, S & & \bullet \ar[r,"l_S",tail] \ar[d] \ar[ll,"{\rm coker}\, \lambda_S"{xshift=-30pt},bend left=30,two heads] \ar[r, phantom, shift right=4ex, "{\rm PB}" marking] & \bullet \ar[d,"{\rm im}\, b",tail]\\
    & C \ar[r] & \bullet \ar[r,"{\rm im}\, g"',tail] & D
\end{tikzcd}
\quad
\begin{tikzcd}
    A \ar[d,"{\rm coim}\, a"',two heads] \ar[r,"{\rm coim}\, f",two heads] \ar[r, phantom, shift right=4ex, "{\rm PO}" marking] & \bullet \ar[r] \ar[d] & B \ar[dd,"b"] &\\
    \bullet \ar[d] \ar[r,"r_S"',two heads] & \bullet \ar[rd,dashed,"\rho_S"] & & {\rm Ker}\, S \ar[ll,"{\rm ker}\, \rho_S"'{xshift=30pt},bend right=30,tail] \\
    C \ar[rr,"g"'] & & D &
\end{tikzcd}
\]
Consider the diagram of two consecutive squares:
\begin{equation}\label{diagram1}
\begin{tikzcd}
    A \ar[r,"f"] \ar[d,"a"'] \ar[r, phantom, shift right=4ex, "S" marking] & B \ar[r,"g"] \ar[d,"b" description] \ar[r, phantom, shift right=4ex, "T" marking] & C \ar[d,"c"]\\
    A' \ar[r,"f'"'] & B' \ar[r,"g'"'] & C'
\end{tikzcd}
\end{equation}

The following theorem holds:

\begin{theorem}\label{lambek_ex_th}
Suppose that in diagram~{\rm \eqref{diagram1}} in a semiexact category $(\mathcalb{C},\mathcalb{N})$, the morphism $b$ is exact and the compositions $gf$ and $g'f'$ are null.
Then there exists a unique morphism $\Lambda\colon \mathrm{Im}\, S \to \mathrm{Ker}\, T$ such that $r_T l_S = (\mathrm{ker}\, \rho_T) \Lambda (\mathrm{coker}\, \lambda_S)$.
\end{theorem}
\begin{proof}
    For brevity, we put $\lambda := \lambda_S$, $l := l_S$, $r := r_T$, and $\rho := \rho_T$.
    Since $b$ is exact, it can be presented as $b = ({\rm im}\, b) ({\rm coim}\, b)$.
    The~existence of pullbacks for pairs of kernels and of pushouts for pairs of cokernels yields the~commutative diagram:
    \begin{equation}\label{diagram2}
    \begin{tikzcd}[column sep=large]
        A \ar[rr,"f"] \ar[dd,"a"'] \ar[rd,"\lambda"',dashed] & & B \ar[r,"{\rm coim}\, g",two heads] \ar[d,"{\rm coim}\, b"',two heads] & {\rm Coim}\, g \ar[r,"({\rm im}\, g) \overline{g}"] \ar[d,"t"] & C \ar[dd,"c"]\\
        & M \ar[r, phantom, shift right=4ex, "{\rm PB}" marking] \ar[r,"l",tail] \ar[d,"s"'] & \bullet \ar[d,"{\rm im}\, b",tail] \ar[r,"r"',two heads] \ar[r, phantom, shift left=4ex, "{\rm PO}" marking] & N \ar[rd,"\rho",dashed] & \\
        A' \ar[r,"\overline{f'} ({\rm coim}\, f')"'] & {\rm Im}\, f' \ar[r,"{\rm im}\, f'"',tail] & B' \ar[rr,"g'"'] & & C'
    \end{tikzcd}
    \end{equation}
    Note that here
    $$
    (rl)\lambda = t \underbrace{({\rm coim}\, g) f}_{{\rm null}},
    \quad
    \rho (rl) = \underbrace{g' ({\rm im}\, f')}_{{\rm null}} s.
    $$
    Hence, there exist $\alpha\colon {\rm Coker}\, \lambda \to N$ and $\beta\colon M \to {\rm Ker}\, \rho$ such that
    $$
    \alpha ({\rm coker}\, \lambda) = rl = ({\rm ker}\, \rho) \beta.
    $$
    Observe that
    $$
    \rho (rl) = \rho \alpha ({\rm coker}\, \lambda),
    \quad
    ({\rm ker}\, \rho) \beta \lambda = (rl) \lambda.
    $$
    Hence, $\rho \alpha$ and $\beta \lambda$ are null because $rl\lambda$ and $\rho rl$ are null, every kernel is an~$\mathcalb{N}$-monomorphism, and every cokernel is an~$\mathcalb{N}$-epimorphism.
    Since $\rho \alpha, \beta \lambda \in \mathcalb{N}$, the~morphisms $\alpha$ and $\beta$ factor through ${\rm ker}\, \rho$ and ${\rm coker}\, \lambda$ respectively:
    $$\alpha = ({\rm ker}\, \rho) \alpha',\quad \beta = \beta' ({\rm coker}\, \lambda).$$
    Since $\alpha ({\rm coker}\, \lambda) = rl = ({\rm ker}\, \rho) \beta$, we have
    $$({\rm ker}\, \rho) \alpha' ({\rm coker}\, \lambda) = rl = ({\rm ker}\, \rho) \beta' ({\rm coker}\, \lambda).$$
    Since ${\rm ker}\, \rho$ is a monomorphism and ${\rm coker}\, \lambda$ is an epimorphism, we have $\alpha' = \beta'$.
    We put $\Lambda := \alpha' = \beta'\colon {\rm Im}\, S \to {\rm Ker}\, T$.
    The final diagram looks like this:
    \begin{equation}\label{full_diagram}
    \begin{tikzcd}[trim right=0cm, trim left=0cm]
        & & & {\rm Coker}\, \lambda \ar[dd,"\alpha" {yshift=25pt},bend left=80,dashed] \ar[lldddd,"\alpha' = \Lambda",dashed,  controls={+(5,-0.5) and +(6,-1.5)},yshift=-0.5pt,xshift=0.5pt] \ar[lldddd,"\Lambda = \beta'"',dashed,  controls={+(-6.5,2) and +(-5.5,-0.5)}] & \\
        A \ar[rr,"f" {xshift=-20pt}] \ar[dd,"a"'] \ar[rd,"\lambda"',dashed] & & B \ar[r,"{\rm coim}\, g",two heads] \ar[d,"{\rm coim}\, b"',two heads] & {\rm Coim}\, g \ar[r] \ar[d,"t",two heads] & C \ar[dd,"c"]\\
        & M \ar[r, phantom, shift right=4ex, "{\rm PB}" marking] \ar[r,"l",tail] \ar[d,"s"',tail] \ar[rruu,"{\rm coker}\, \lambda" {yshift=5pt,xshift=10pt}, bend left=35,two heads] \ar[dd,"\beta"' {yshift=-25pt},bend right=80,dashed] & \bullet \ar[d,"{\rm im}\, b",tail] \ar[r,"r"',two heads] \ar[r, phantom, shift left=4ex, "{\rm PO}" marking] & N \ar[rd,"\rho",dashed] & \\
        A' \ar[r] & {\rm Im}\, f' \ar[r,"{\rm im}\, f'"',tail] & B' \ar[rr,"g'"' {xshift=25pt}] & & C'\\
        & {\rm Ker}\, \rho \ar[rruu,"{\rm ker}\, \rho"' {yshift=-5pt,xshift=-10pt}, bend right=35,tail,yshift=0.5pt] & & &
    \end{tikzcd}
    \end{equation}
\end{proof}

We refer to the morphism $\Lambda\colon {\rm Im}\, S \to {\rm Ker}\, T$ constructed in Theorem~\ref{lambek_ex_th} as the \emph{Lambek morphism} of the squares $S$ and $T$.

\begin{theorem}[The~Lambek isomorphism]
    Suppose that in diagram~{\rm \eqref{diagram1}} in a homological category $(\mathcalb{C},\mathcalb{N})$, $b$ is an exact morphism and the rows are exact.
    Then the Lambek morphism $\Lambda\colon \mathrm{Im}\, S \to \mathrm{Ker}\, T$ is an isomorphism.
\end{theorem}
\begin{proof}
    Consider diagram~(\ref{full_diagram}) from the proof of Theorem~\ref{lambek_ex_th}.
    By Lemma~\ref{ker_coker_pb_po}, $l$ is a kernel, moreover, $l = {\rm ker}(({\rm coker}\, f') ({\rm im}\, b))$. Note that
    $$({\rm coker}\, f')({\rm im}\, b)({\rm coim}\, b)({\rm im}\, f) \in \mathcalb{N}.$$
    Hence, there exists a unique morphism $m\colon {\rm Im}\, f \to M$ such that $lm = ({\rm coim}\, b) ({\rm im}\, f)$.
    Dually, the morphism $r = {\rm coker}(({\rm coim}\, b) ({\rm ker}\, g))$ in (\ref{full_diagram}) is a cokernel and there exists a unique morphism $n\colon N \to {\rm Coim}\, g'$ such that $nr = ({\rm coim}\, g') ({\rm im}\, b)$.
    \begin{equation}\label{diagram3}
    \begin{tikzcd}[column sep=large]
        A \ar[r,"\overline{f} ({\rm coim}\, f)"] \ar[dd,"a"'] \ar[rd,"\lambda"'] & {\rm Im}\, f \ar[r,"{\rm im}\, f",tail] \ar[d,dashed,"m"] & B \ar[d,"{\rm coim}\, b"',two heads] \ar[r, "{\rm coim}\, g", two heads] & {\rm Coim}\, g \ar[d,"t",two heads] \ar[r,"({\rm im}\, g) \overline{g}"] & C \ar[dd,"c"]\\
        & M \ar[r, phantom, shift right=4ex, "{\rm PB}" marking] \ar[r,"l",tail] \ar[d,"s"',tail] & \bullet \ar[d,"{\rm im}\, b",tail] \ar[r,"r"',two heads]  \ar[r, phantom, shift left=4ex, "{\rm PO}" marking] & N \ar[d,"n"',dashed] \ar[rd,"\rho"] \\
        A' \ar[r,"\overline{f'} ({\rm coim}\, f')"'] & {\rm Im}\, f' \ar[r,"{\rm im}\, f'"',tail] & B' \ar[r,"{\rm coim}\, g'"', two heads] & {\rm Coim}\, g' \ar[r,"({\rm im}\, g') \overline{g'}"'] & C'
    \end{tikzcd}
    \end{equation}
    Note that $l = {\rm ker}(({\rm coker}\, f') ({\rm im}\, b)) = {\rm ker}(({\rm coim}\, g') ({\rm im}\, b))$ because the rows are exact.
    Hence, by commutativity, $l = {\rm ker}(nr)$.
    Dually, $r = {\rm coker}(lm)$.
    Note that $l$ is exact since it is a pullback of a kernel ${\rm im}\, f'$.
    Dually, $n$ is exact.
    Since~\eqref{diagram3} commutes, ${\rm coker}\, \lambda = {\rm coker}(m \overline{f} ({\rm coim}\, f))$.
    Therefore, ${\rm coker}\, \lambda = {\rm coker}\, m$.
    By~duality, ${\rm ker}\, \rho = {\rm ker}\, n$.
    Consider the composition $lm$.
    By Lemma~\ref{composition_lemma}, there exists an exact morphism $\varepsilon\colon {\rm Coker}\, m \to {\rm Coker}(lm)$ such that $({\rm coker}(lm)) l = \varepsilon ({\rm coker}\, m)$.
    Finally, we get the equalities
    $$rl = ({\rm coker}(lm)) l = \varepsilon ({\rm coker}\, m) = \varepsilon ({\rm coker}\, \lambda) = \alpha ({\rm coker}\, \lambda)$$
    where $\alpha\colon {\rm Coker}\, \lambda \to N$ is the same as in~\eqref{full_diagram}.
    Dually, consider the composition $nr$.
    By Lemma~\ref{composition_lemma}, there exists an exact morphism $\psi\colon {\rm Ker}(nr) \to {\rm Ker}\, n$ such that $r ({\rm ker}(nr))= ({\rm ker}\, n) \psi$ and the following equalities hold:
    $$rl = r ({\rm ker}(nr)) = ({\rm ker}\, n) \psi = ({\rm ker}\, \rho) \psi = ({\rm ker}\, \rho) \beta$$
    where $\beta\colon M \to {\rm Ker}\, \rho$ is as in~\eqref{full_diagram}.
    Thus, $\varepsilon = \alpha$, $\psi = \beta$, and $\alpha$ and $\beta$ are exact morphisms.
    By the definition of $\Lambda$, we have
    $$({\rm ker}\, \rho) \Lambda = \alpha = \varepsilon,\ \Lambda ({\rm coker}\, \lambda) = \beta = \psi.$$
    Since $({\rm coker}\, m) ({\rm ker}\, l) \in \mathcalb{N}$, by Lemma~\ref{composition_lemma}(2), $\varepsilon$ is a kernel, and hence, by Lemma~\ref{ker-cancelation}, $\Lambda$ is a kernel.
    Dually, $\psi$ is a cokernel, and thus $\Lambda$ is a cokernel.
    Therefore, $\Lambda$ is an isomorphism.
\end{proof}

\end{document}